\documentclass[11pt]{article}

\usepackage[english]{isodate}
\usepackage{graphicx}
\usepackage{parskip,amsfonts}
\usepackage[numbib]{tocbibind}
\usepackage[margin=1.3in]{geometry}
\usepackage{hyperref}
\usepackage{amsthm, amsmath}
\usepackage[shortlabels]{enumitem}
\usepackage{listings}
\usepackage{algpseudocode,algorithm}
\usepackage[thinlines]{easytable}
\usepackage{tkz-graph}
\usepackage[htt]{hyphenat}
\usetikzlibrary{calc}

\newtheorem{theorem}{Theorem}[section]
\newtheorem{proposition}{Proposition}[section]

\renewcommand{\O}{\text{O}}

\begin{document}

\begin{titlepage}
  \begin{center}
    {\Large \textbf{Decomposing Linear Representations of Finite Groups}}

    \vspace{3cm}

    \includegraphics[width=32mm]{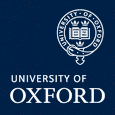}

    \vspace{3cm}

    {\large
    Kaashif Hymabaccus \par
    Wadham College \par
    University of Oxford \par
    \vspace{1cm}
    Supervisor: Dr Dmitrii Pasechnik
    }

    \vspace{3cm}

    \emph{Final Honour School of Mathematics and Computer Science Part C} \par
    Trinity 2019 \par

  \end{center}
\end{titlepage}

\newpage

\section*{Abstract}

We develop a package using the computer algebra system GAP for
computing the decomposition of a representation $\rho$ of a finite
group $G$ over $\mathbb{C}$ into irreducibles, as well as the
corresponding decomposition of the centraliser ring of
$\rho(G)$. Currently, the only open-source programs for decomposing
representations are for non-zero characteristic fields. While methods
for characteristic zero are known, there are no open-source computer
programs that implement these methods, nor are details on how to
achieve good performance of such an implementation published. We aim
to record such details and demonstrate an application of our program
in reducing the size of semidefinite programs.

\newpage

\tableofcontents
\newpage

\section{Introduction} \label{sec:introduction}
\subsection{Motivation and Requirements}

In 1971, in his graduate textbook on the linear representations of
finite groups \cite{serre_2014}, Serre specified a method of computing
the decomposition of a representation of a finite group $G$ into
irreducible subrepresentations. Despite the existence of Serre's text,
this algorithm has no open source implementation. Indeed, there is no
open source program solving this problem in general at all. This is
the problem our project aims to solve, using the computer algebra
systems GAP \cite{GAP4} and SageMath \cite{sagemath}.

Specifically, we aim to provide a GAP package providing the
functionality to compute the following, given a finite dimensional
representation of a finite group $\rho : G \to \text{GL}(V)$, where
$V$ is a $\mathbb{C}$-vector space.

\begin{itemize}
\item A decomposition of $V$ into irreducible subrepresentations
\item The corresponding block diagonalisation of $\rho$, along with
  the associated basis change matrix
\item A basis for the centraliser ring $C_\rho \subseteq
  \text{End}(V)$
\end{itemize}

We only deal with the case where the representation $\rho$ is over
$\mathbb{C}$ - in particular, this means we are in characteristic
zero. In the case where the representation is over a finite field, a
user could take advantage of the Meataxe algorithms, introduced by
Parker \cite{parker_1984} and later improved by various authors. These
algorithms have been implemented in GAP and allow computations of
decompositions of modules, tests of irreducibility, isomorphisms
between modules and so on. None of these methods are applicable in the
case of a characteristic zero field, which is what we focus on in this
project.

We also aim to extensively test and document our package to ensure
correctness and ease of use. After the completion of the project, our
goal is for this package to be included with the GAP distribution.

\subsection{Outline}

In Section \ref{sec:background}, we provide some background on
representation theory, group theory, and some terminology (some of
which is not standardised) that will be used throughout this project.

In Section \ref{sec:serre}, we describe Serre's algorithms to
decompose a representation of a finite group, with some discussion of
performance and possible improvements.

In Section \ref{sec:mymethod}, we describe an alternate algorithm of
our own design, incorporating some results due to Serre and various
small performance enhancements.

In Section \ref{sec:performance}, we describe our testing and
benchmarking methodology. We analyse the performance, focusing on the
effect of the size of $G$ and the degree of the representation $\rho$
on the running time, as well as measuring the effect of our
optimisations.

In Section \ref{sec:crossing}, we describe an application of our
program to reducing the dimension of semidefinite programs, using a
method due to de Klerk et al. \cite{deKlerk2007}. We reproduce their
calculation of a bound on the crossing number of $K_{m,n}$ (a complete
bipartite graph) using our method to formulate the program.

In Section \ref{sec:conclusion}, we describe how our contributions
satisfy the requirements, possible improvements, and further work that
could be done on this project.

In the appendix Section \ref{sec:unitary}, we implement an algorithm
to unitarise representations using methods developed while
implementing the main algorithms of this project. We also describe an
algorithm due to Dixon \cite{dixon_1970} for decomposing unitary
representations that has some desirable properties. This section is
not a part of the requirements set out earlier, but is a possibly
useful consequence of our work.

\section{Background} \label{sec:background}

In this section, we provide some background information required to
understand the results presented in this project. A basic
understanding of group theory and linear algebra is assumed, but
nothing more.

\subsection{Representation Theory} \label{sec:reptheory}

\newcommand{\GL}{\mbox{GL}}

Let $G$ be a finite group. A \emph{representation} of $G$ is a
homomorphism $\rho$ from $G$ to the automorphism group of a vector
space, $\GL(V)$. We will only consider the case where $V$ is a
finite-dimensional vector space over $\mathbb{C}$.

A common abuse of terminology is to refer to $V$, the vector space, as
a representation of $G$. This is done only when it is clear what the
action of $G$ is.

Let $\rho : G \to \GL(V)$ be a representation of $G$.

$\rho$ is \emph{isomorphic as a representation} to $\tau : G \to
\GL(W)$ if there is a linear isomorphism $\alpha : V \to W$ such that
for all $g \in G$, $\alpha \circ \rho(g) = \tau(g) \circ \alpha$. In
other words, $\rho$ and $\tau$ differ by a change of basis.

A \emph{subrepresentation} of $\rho$ is a representation $\rho|_W : G
\to \GL(W)$, where $W$ is a subspace of $V$ which is invariant under
the action of $G$. The action on $W$ is given by $\rho|_W(g)w =
\rho(g)|_W w$. This is well defined because $W$ is preserved by
$\rho(g)$.

A representation $\rho$ is said to be \emph{irreducible} if it has no
subrepresentations other than $0$ and $V$. Note that, for finite group
representations over $\mathbb{C}$, this is equivalent to saying that
$V$ does not break down into a direct sum of subrepresentations
\cite{serre_2014}. We will sometimes refer to an ``irreducible
representation'' as an ``irreducible''.

Let $W_1, \ldots, W_n$ be the complete list of irreducibles of $G$. If
$V = \bigoplus\limits_{i=1}^m U_i$ is the decomposition of $V$ into
irreducibles, then the \emph{canonical decomposition} of $V$ is
$\bigoplus\limits_{k=1}^n V_k$ where each $V_k$ is the direct sum of the $U_i$
which are isomorphic as representations to $W_k$.

The \emph{centraliser} (sometimes called the \emph{commutant}) of a
representation $\rho$ is the $\mathbb{C}$-vector space of linear maps
$A$ such that $A\rho(g) = \rho(g)A$ for all $g \in G$.

Given $\rho$, the dual representation $\rho^*$ is the
representation defined by:
\[
\rho^*(g) = \rho(g^{-1})^T
\]

\subsection{Orbitals}

A more detailed presentation of some of these definitions, with
illustrative examples, can be found in P. Cameron's \emph{Permutation
  Groups} \cite{cameron_1999}.

Frequently, it will be convenient to consider a representation $\rho :
G \to \GL(V)$ where all of the images $\rho(g)$ are permutation
matrices as a map $\rho : G \to \mbox{Sym}(X)$, where $X$ is a basis
for $V$. This is a group action of $G$ on $X$.

An \emph{orbital} of this action is an orbit of $G$ acting on $X
\times X$ by: $g \cdot (x,y) = (g \cdot x, g \cdot y)$.

The \emph{orbital (di)graph} associated with an orbital $\Delta$ is
the directed graph with vertex set $X$ and directed edge set $\Delta$.

\begin{center}
\begin{figure}[h]
\begin{center}
\begin{minipage}{0.3\textwidth}
\begin{tikzpicture}
  \Vertex{2}
  \WE(2){1}
  \EA(2){3}
  \tikzset{EdgeStyle/.append style = {->}}
  \Loop[dist=1cm, dir=NO](1.north)
  \Loop[dist=1cm, dir=NO](2.north)
  \Loop[dist=1cm, dir=NO](3.north)
  \draw[dotted,red] ($(current bounding box.south east)+(6pt,-6pt)$) rectangle ($(current bounding box.north west)+(-6pt,6pt)$);
\end{tikzpicture}
\end{minipage}
\begin{minipage}{0.3\textwidth}
\begin{tikzpicture}
  \Vertex{2}
  \SOWE(2){1}
  \SOEA(2){3}
  \tikzset{EdgeStyle/.append style = {->, bend left = 10}}
  \Edge(1)(2)
  \Edge(2)(1)
  \Edge(1)(3)
  \Edge(3)(1)
  \Edge(2)(3)
  \Edge(3)(2)
  \draw[dotted,red] ($(current bounding box.south east)+(6pt,-6pt)$) rectangle ($(current bounding box.north west)+(-6pt,6pt)$);
\end{tikzpicture}
\end{minipage}
\end{center}
\caption{$S_3$ acting on $\{1,2,3\}$ has two orbital digraphs}
\end{figure}
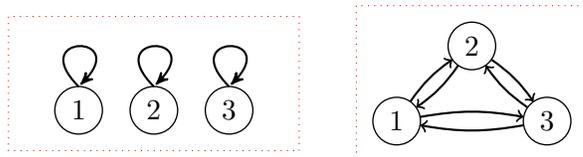
\end{center}

When we refer to an orbital as a matrix, this refers to the adjacency
matrix of the orbital graph.

Computing with the full adjacency matrices is inconvenient due to
their space inefficiency. It is more convenient to work with collapsed
versions of these matrices, which allow computations to be done much
faster.

Let $\alpha \in X$ and fix an ordering of the orbits of $G_\alpha$
(the stabiliser of $\alpha$): $X_1 = \{ \alpha \}, X_2, \ldots,
X_r$. Choose representatives $\alpha_i \in X_i$. Let $\Gamma = (X,
\Delta)$ be an orbital graph.

A \emph{collapsed adjacency matrix} (in the sense of Praeger and
Soicher \cite{praeger_soicher_1996}) for $\Gamma$, with respect to the
choice of $\alpha$, representatives, and orderings, is the $r \times r$
integer matrix $A$ defined by:

\[
A_{ij} = |\Gamma(\alpha_i) \cap X_j|
\]

That is, $A_{ij}$ is the number of neighbours of $\alpha_i$ in the
graph $\Gamma$ which are contained in $X_j$.

These collapsed matrices are useful because they provide a natural
isomorphism from the algebra generated by orbital matrices $X$ to a
smaller-dimensional algebra $Y$, given by the collapsed orbital
matrices. This means that whenever we have a problem in $X$ that can
be formulated in terms of spectra of matrices, we can pass to $Y$
using the isomorphism and solve the same problem in $Y$. A detailed
explanation of the isomorphism can be found in \cite{deKlerk2007}, and
an abbreviated explanation in Section \ref{sec:crossing}, where we use
this method to reduce the dimension of a semidefinite program.

\subsection{Cyclotomic numbers} \label{sec:cyclo}

\emph{Cyclotomic numbers} are the numbers $z \in \mathbb{C}$ such that
$z \in \mathbb{Q}(\zeta_n)$ for some $n$, where $\zeta_n$ is a complex
primitive $n$th root of unity. $\mathbb{Q}(\zeta_n)$ is called a
\emph{cyclotomic field}.

Brauer \cite{brauer_1945} proved that every irreducible complex
character of a finite group $G$ can be realised by a representation of
$G$ over the cyclotomic field $\mathbb{Q}(\zeta_x)$, where $x$ is the
exponent of $G$ (the smallest $x$ such that for all $g \in G$, $g^x =
1_G$). This means that if we restrict our attention to representations
with cyclotomic coefficients, we get all representations of $G$, up to
isomorphism.

In GAP (the computer algebra system that we use), we can represent
general algebraic numbers, but the algorithms for performing
arithmetic with these is not efficient when compared with cyclotomic
numbers. Due to Brauer's theorem, we do not need to consider
non-cyclotomic algebraic numbers.

Throughout this project, when we take a representation over
$\mathbb{C}$, we assume this representation is over the cyclotomic
numbers.

\subsection{Computer algebra systems used}

To implement this project, we used two computer algebra systems: GAP
\cite{GAP4} and SageMath \cite{sagemath}.

The main part of the project is a GAP package, implemented in the GAP
programming language, which is designed to be easy to read for
mathematicians and programmers alike. We will refrain from using
non-obvious features in code snippets without explanation. A reference
for the programming language can be found here:
\url{https://www.gap-system.org/Manuals/doc/ref/chap4.html}.

Some parts of the project are implemented using SageMath, which uses
the Python programming language. This was necessary since certain
algorithms are only implemented in SageMath, like the computation of a
complete list of irreducibles of $S_n$ using integer matrices. SageMath also
has a convenient interface to both GAP and a semidefinite program
solver, making it an ideal choice to implement Section
\ref{sec:crossing}, which explores how our project can be used to
speed up computations by finding the optimal block diagonalization for
a semidefinite program.

\section{Algorithms due to Serre} \label{sec:serre}

The algorithms and theorems in this section are taken partially from
Serre's text on representation theory \cite{serre_2014}. In that text,
you will find a more mathematically rigorous treatment of this
material, including full proofs of correctness. Our focus is on
algorithmic details and performance, which Serre was not overly
concerned with.

We first present the basic algorithm as is presented in Serre's
text. Then, we describe optimisations that greatly reduce the running time
in certain cases.

\subsection{Basic algorithm} \label{sec:basic}

The algorithm proceeds in two steps. Given a representation $\rho : G
\to \GL(V)$, we first compute the canonical decomposition of $V$, then
we decompose each summand into irreducibles.

Let $W_1, \ldots, W_h$ be the complete list of irreducibles of $G$, $n_1,
\ldots, n_h$ their degrees and $\chi_1, \ldots, \chi_h$ their
characters.

We use a result due to Serre \cite{serre_2014}:
\vspace{1em}
\begin{theorem}
The projection $p_i$ of $V$ onto $V_i$ associated with the canonical
decomposition is given by:
$$p_i = \frac{n_i}{|G|} \sum_{t \in G} \overline{\chi_i(t)} \rho(t)$$
\end{theorem}

We can iterate over each irreducible and compute the image of the projection
to obtain our canonical decomposition $V = \bigoplus\limits_i V_i$.

The next step is to decompose each $V_i$ into a direct sum of
irreducible subrepresentations. The key result we use to compute this
decomposition is due to Serre \cite{serre_2014}:
\vspace{1em}
\begin{theorem} \label{thm:prop8}
  Let $n$ be the degree of the irreducible $W_i$, with $W_i$ given in matrix
  form by $r_{\alpha\beta}(s)$. Let $p_{\alpha\beta}$ denote the
  linear map $V \to V$ given by:

  $$p_{\alpha\beta} = \frac{n}{|G|} \sum_{t \in G} r_{\beta\alpha}(t^{-1})\rho(t)$$

  \begin{enumerate}[(a)]
    \item The map $p_{\alpha\alpha}$ is a projection; it is zero on
      the $V_j$ for $j \neq i$. Its image $V_{i, \alpha}$ is contained
      in $V_i$ and $V_i$ is the direct sum of the $V_{i, \alpha}$ for
      $1 \leq \alpha \leq n$. We have $p_i = \sum_\alpha
      p_{\alpha\alpha}$.

    \item The linear map $p_{\alpha\beta}$ is zero on the $V_j$ for $j
      \neq i$, as well as on the $V_{i, \gamma}$ for $\gamma \neq
      \beta$; it defines an isomorphism from $V_{i, \beta}$ onto
      $V_{i, \alpha}$.

    \item Let $x_1$ be an element $\neq 0$ of $V_{i, 1}$ and let
      $x_\alpha = p_{\alpha 1}(x_1) \in V_{i, \alpha}$. The $x_\alpha$
      are linearly independent and generate a vector subspace $W(x_1)$
      stable under $G$ and of dimension $n$. For each $s \in G$ we
      have:
      $$\rho(s)(x_\alpha) = \sum_\beta r_{\beta\alpha}(s)x_\beta$$
      (in particular, $W(x_1)$ is isomorphic to $W_i$).

    \item If $(x_1^{(1)}, \ldots, x_1^{(m)})$ is a basis of $V_{i,
      1}$, the representation $V_i$ is the direct sum of the
      subrepresentations $(W(x_1^{(1)}), \ldots, W(x_1^{(m)}))$.
  \end{enumerate}
\end{theorem}

We will iterate over the canonical summands $V_i$, compute a basis for
$V_{i,1}$, $(x_1^{(1)}, \ldots, x_1^{(m)})$, and for each $j$, compute
the vector subspace $W(x_1^{(j)})$. Then $V_i = \bigoplus\limits_j W(x_1^{(j)})$
is the irreducible decomposition of this canonical summand.

The full algorithm is as follows:

\begin{algorithm}[h]
\begin{algorithmic}
\Function{decompose representation}{$\rho : G \to \GL(V)$}
\State $n \gets \text{degree}(\rho)$
\State $\Delta \gets \{\}$ \Comment{This is where we will build the decomposition}
\For{$\rho_i \in \text{irreducibles}(G)$}
  \State $\Delta_i \gets \{\}$ \Comment{Decomposition of $V_i$}
  \State $n_i \gets \text{degree}(\rho_i)$
  \State $\chi_i \gets \text{character}(\rho_i)$
  \State $p_i \gets \frac{n_i}{|G|} \sum_{t \in G} \overline{\chi_i(t)} \rho(t)$
  \State $V_i \gets p_i(V)$
  \For{$1 \leq \alpha, \beta \leq n_i$}
    \State $p_{\alpha\beta} \gets = \frac{n_i}{|G|} \sum_{t \in G} (\rho_i(t^{-1}))_{\beta\alpha}\rho(t)$
  \EndFor
  \State $V_{i,1} \gets p_{11}(V_i)$
  \For{$x_1^{(j)} \in \text{basis}(V_{i,1})$}
    \State $W(x_1^{(j)}) \gets \text{span}(\{p_{\alpha1}(x_1^{(j)}) : 1 \leq
      \alpha \leq n_i\})$
    \State $\Delta_i \gets \Delta_i \cup \{W(x_1^{(j)})\}$
  \EndFor
  \State $\Delta \gets \Delta \cup \Delta_i$
\EndFor
\State \Return $\Delta$
\EndFunction
\end{algorithmic}
\end{algorithm}

This basic version of the algorithm can be called in the GAP package
as the function {\tt IrreducibleDecomposition(rho : no\_optimisations)}.

\subsection{Optimised algorithm}

The most obvious optimisation is to prune the set of irreducibles to the
$W_i$ that actually appear in the decomposition before performing any
expensive per-irreducible computations. Let $\chi_\rho$ be the character of
$\rho$ and recall that the irreducible characters form an orthonormal
basis for the space of characters. We can restrict our attention to
the irreducibles $W_i$ such that $\langle \chi_\rho, \chi_i \rangle > 0$,
since these $W_i$ are the only ones that appear with nonzero
multiplicity in the decomposition of $\rho$.

This is a cheap calculation since we can reduce it to a summation over
the classes of $G$. Suppose the conjugacy classes of $G$ are given by
$\{t_1^G, \ldots, t_m^G\}$. Then:

\[
\langle \chi_\rho, \chi_i \rangle = \frac{1}{|G|} \sum_{g \in G}
\chi_\rho(g) \chi_i(g)^* = \frac{1}{|G|} \sum_j |t_j^G| \chi_\rho(t_j) \chi_i(t_j)^*
\]

This will reduce the amount of work if only a small number of irreducibles
actually appear in $\rho$.

\subsubsection{Computing the canonical decomposition} \label{sec:canonical} \label{sec:use_basis}

In one case, we can skip the canonical decomposition computation
altogether. If we find only one irreducible $W_i$ appears in the
decomposition of $V$, the canonical decomposition is just the whole
space $V$.

Otherwise, we must compute the projections $p_i : V \to V_i$. The main
optimisation for this step requires some extra information to be
provided by the user.

Suppose we are given a basis $B_1, \ldots, B_d$ for the centraliser
ring $C$ of $\rho$, which is orthonormal with respect to the inner
product $\langle A, B \rangle = \text{Trace}(AB^*)$, where $B^*$ is
the conjugate transpose of $B$. We can rewrite the expression for
$p_i$ as follows. Let $\{ t_1, \ldots, t_m \}$ be a set of
representatives for the conjugacy classes of $G$. Then:

\begin{align*}
p_i &= \frac{n_i}{|G|} \sum_{t \in G} \overline{\chi_i(t)} \rho(t) \\
&= \frac{n_i}{|G|} \sum_j \sum_{s \in t_j^G}
\overline{\chi_i(s)} \rho(s) \\
&= \frac{n_i}{|G|} \sum_j \overline{\chi_i(t_j)}
\sum_{s \in t_j^G}  \rho(s)
\end{align*}

A matrix $M$ is in the centraliser of $\rho$ exactly when conjugating
$M$ by any $\rho(g)$ leaves $M$ unchanged. Say $S = t^G$, the
conjugacy class of $t$ in $G$. Then $\sum_{s \in S} \rho(s)$ is an
element of $C$, since:

$$\rho(g) \left( \sum_{s \in S}\rho(s) \right) \rho(g)^{-1} = \sum_{s
  \in S}\rho(gsg^{-1}) = \sum_{s \in gSg^{-1}}\rho(s) = \sum_{s \in S}
\rho(s)$$

Now we can write $\sum_{s \in S} \rho(s)$ as a vector in the space
spanned by the $B_k$, using the inner product.

We also require that $B_k^*$ is in the centraliser.

\begin{equation} \label{eqn:classsumtrick}
\begin{aligned}
\left\langle \sum_{s \in S} \rho(s), B_k \right\rangle &= \sum_{s \in S} \langle
\rho(s), B_k \rangle \\
&= \sum_{s \in S} \text{Trace}(\rho(s)B_k^*) \\
&= \sum_{s \in S} \text{Trace}(\rho(g_s)\rho(t)\rho(g_s)^{-1}B_k^*)
& \text{($s = g_s t g_s^{-1}$ for some $g_s \in G$)} \\
&= \sum_{s \in S} \text{Trace}(\rho(g_s)\rho(t)B_k^*\rho(g_s)^{-1})
& \text{(we require $B_k^* \in C$)}\\
&= \sum_{s \in S} \text{Trace}(\rho(t)B_k^*) \\
&= |S| \text{Trace}(\rho(t)B_k^*)
\end{aligned}
\end{equation}

Let $\{t_1, \ldots, t_n\}$ be representatives of the conjugacy classes
of $G$. Now, to sum over the whole group, we need to compute:

\begin{equation} \label{eqn:groupclasssum}
\sum_{g \in G} \rho(g) = \sum_{i=1}^n \sum_{j=1}^d |t_i^G| \text{Trace}(\rho(t_i)B_j^*)
\end{equation}

The number of field operations required to compute the summation
(\ref{eqn:groupclasssum}) depends on both $d$, the dimension of the
centraliser, and $n$, the number of conjugacy classes of $G$, but does
not depend on $|G|$ at all.

This proof depends on the fact that $C$ is closed under complex
conjugation. That is, to use this trick, we require that $C$ is a
matrix $*$-algebra.

There are several cases where $C$ has this property. The most obvious
is if we have a $\rho : G \to M_n(\mathbb{C})$ that has already been
block-diagonalised according to its irreducible decomposition.

Let $\rho = \bigoplus\limits_{i=1}^n \bigoplus\limits_{j=1}^{m_i} \rho_i$ where the
$\rho_i$ are distinct and irreducible, and $m_i$ is the multiplicity
of $\rho_i$ in $\rho$. Recall that the centraliser of
$M_n(\mathbb{C})$ is the set of scalar matrices $\lambda I$ for
$\lambda \in \mathbb{C}$.

In our case, each block of our matrix $\rho(g)$ (corresponding to
canonical summand $V_i$) is a block-diagonal matrix
$\text{diag}(\rho_i(g), \ldots, \rho_i(g))$. If we now consider the
ring of coefficients in our matrix to be $M_{d_i}(\mathbb{C})$, where
$d_i$ is the degree of $\rho_i$, then this is a scalar matrix
$\rho_i(g) I_{m_i}$, so commutes with all matrices in $M_{m_i}(S)$
where $S$ is the subring of $M_{d_i}(\mathbb{C})$ consisting of scalar
matrices. We know that the $\rho_i$ do not commute with any more
matrices due to Schur's lemma, which tells us that the only
$G$-invariant linear endomorphisms of an irreducible are scalar.

To clarify, the centraliser of $\rho$ is spanned as a vector space by
matrices with $n$ scalar matrix blocks indexed by $i$, with block
$B_i$ (corresponding to $V_i$) having size $m_id_i \times
m_id_i$. Each of these blocks $B_i$ is divided into an $m_i \times
m_i$ grid of smaller blocks, each of size $d_i \times d_i$. An
orthonormal basis for $C$ is given by a set of $\sum_i m_i^2$
matrices, each with exactly one small block set to the identity matrix
and all other blocks set to zero.

For example, if we have a representation $\rho = \rho_1 \oplus
2\rho_2$, where $\deg \rho_1 = 3$ and $\deg \rho_2 = 2$, then a basis
for $C_\rho$ is:

\[
\left(
\begin{array}{c|cc}
  I_3 & 0 & 0 \\
  \hline
  0 & 0 & 0 \\
  0 & 0 & 0
\end{array}
\right),
\left(
\begin{array}{c|cc}
  0 & 0 & 0 \\
  \hline
  0 & I_2 & 0 \\
  0 & 0 & 0
\end{array}
\right),
\left(
\begin{array}{c|cc}
  0 & 0 & 0 \\
  \hline
  0 & 0 & I_2 \\
  0 & 0 & 0
\end{array}
\right),
\left(
\begin{array}{c|cc}
  0 & 0 & 0 \\
  \hline
  0 & 0 & 0 \\
  0 & I_2 & 0
\end{array}
\right),
\left(
\begin{array}{c|cc}
  0 & 0 & 0 \\
  \hline
  0 & 0 & 0 \\
  0 & 0 & I_2
\end{array}
\right)
\]

It is now clear that given a matrix in the span of these matrices, its
conjugate transpose is also in the span.

In the case that we are not given a basis for $C$, we must calculate
one. In general, this may be difficult. In the specific case where
$\rho$ is an isomorphism to a permutation group, i.e. the permutation
representation of a faithful action of $G$ on some finite set $X$, we
have a basis for $C$ available to use already: the set of orbital
matrices (see Section \ref{sec:background} for their definition).

Let the set of orbitals be given by $\{ \Delta_i : 1 \leq i \leq r
\}$, with $\Delta_i$ having matrix $A_i$. The $A_i$ form an
orthonormal basis for $C$, and can be used with the method of
computing the sum described earlier (see equation
(\ref{eqn:classsumtrick})) since they form a matrix $\ast$-algebra.

Notice that $\{ A_i : 1 \leq i \leq r \}$ span a $\ast$-algebra. This
is because each orbital $\Delta_i$ has a paired orbital $\Delta_{i^*}
= \{ (y,x) : (x,y) \in \Delta_i \}$. It is clear that the associated
adjacency matrices are transposes of each other. This means that the
algebra generated is closed under $\ast$.

A permutation representation is a specific case of a more general
class of representations that work with this trick: \emph{unitary}
representations. A representation $\rho$ is \emph{unitary} if
$\rho(g)^* = \rho(g^{-1})$. Then if $A \in C_\rho$, $A^* \in C_\rho$
also:
\[
\rho(g)A^* = (\rho(g)A^*)^{**} = (A\rho(g)^*)^* = (A\rho(g^{-1}))^* =
(\rho(g^{-1})A)^* = A^* \rho(g^{-1})^* = A^* \rho(g)
\]

When implementing the canonical decomposition, we only use this
optimisation when $\rho$ is unitary. Although the trick also works if
$\rho$ is already block diagonalised, it is a waste of time to check
for this: representations are almost never already optimally block
diagonal unless they were constructed that way intentionally.

Note that for groups with small numbers of generators that have easily
calculated inverses (e.g. $S_n$), checking whether a representation is
unitary is fairly cheap since we only need to check the generators
have unitary images. More precisely, for a single generator, the
calculation is quadratic in the degree of the representation, since we
must transpose the matrix $\rho(g)$, conjugate each cell and check
equality with $\rho(g^{-1})$. This cost is small compared to the
almost cubic running time of a matrix multiplication.

\subsubsection{Unsuitability of the Gram-Schmidt process} \label{sec:gram}

In the previous section, we only attempt to find bases for the
centraliser ring that are already orthogonal. An alternative approach
would be to find a basis that is not necessarily orthogonal and simply
apply the Gram-Schmidt orthonormalisation process to obtain an
orthonormal basis.

\begin{figure}[h] \label{fig:gram}
  \caption{An implementation of the Gram-Schmidt process in GAP}
  \begin{lstlisting}[basicstyle=\ttfamily]
OrthonormalBasis@ := function(v, prod)
    local proj, N, u, e, k;

    proj := function(u, v)
        return (prod(u, v) / prod(u, u)) * u;
    end;

    N := Length(v);

    u := [1..N];
    e := [1..N];

    for k in [1..N] do
        u[k] := v[k] - Sum([1..k-1], j -> proj(u[j], v[k]));
        e[k] := (1/Sqrt(prod(u[k], u[k]))) * u[k];
    od;

    return e;
end;
  \end{lstlisting}
\end{figure}

This algorithm is simple and does not have an unmanageable complexity
in terms of the number of arithmetic operations - only $2nk^2$
operations are required, where $k$ is the number of vectors (in our
case, the dimension of the centraliser) and $n$ is the dimension of
the vectors (in our case, the square of the degree of the
representation). The issue is that the arithmetic operations
themselves are do not happen in constant time or space if we are using
exact numbers, expressed in terms of cyclotomic numbers.

In GAP, a cyclotomic number $z \in \mathbb{C}$, is represented
internally as a list of coefficients $a_i \in \mathbb{Q}$ such that
$\sum_{i=0}^n a_i \zeta_n^i = z$ for some $n$. This means that
cyclotomic numbers can cause a blowup in the amount of space required
to store the values of vectors and consequently, a blowup in the
amount of time required to perform arithmetic operations.

For example, suppose we are given the vector $v = (19, 10)$, $\Vert v
\Vert_2^2 = 461$. If we would like to normalise $v$, we have to
multiply by $\frac1{\sqrt{461}}$, but the smallest cyclotomic field of
which $\frac1{\sqrt{461}}$ is a member is $\mathbb{Q}(\zeta_{461})$,
so the scalars required to represent $\frac{v}{\sqrt{461}}$ will be
internally represented as lists of 461 coefficients. Arithmetic
operations on the normalised vector will thus be much slower than on
the original vector.

The Gram-Schmidt orthogonalisation algorithm does not take this into
consideration and is unusable in practice when dealing with exact
complex numbers.

In general, we try to avoid increasing the degree of the cyclotomic
field we are required to work in. This means we allow the usual field
operations, but not radicals, which is the issue our example
illustrates.

\subsubsection{Computing the irreducible
  decomposition} \label{sec:irr_decomp} \label{sec:groupsum}

In this section, we rephrase Serre's formula for the projections
$p_{\alpha\beta}$ using tensor products, allowing us to optimise the
sums using several observations about summing a representation.

Let $\rho_i$ be the irreducible whose canonical summand we are decomposing,
let $r$ give $\rho_i$ in matrix form, with $n = \deg \rho_i$, then:
\[
p_{\alpha\beta} = \frac{n}{|G|} \sum_{t \in G} r_{\beta\alpha}(t^{-1})\rho(t)
\]

Serre's method then proceeds by computing various images of vectors
under these maps to construct bases for the irreducible subspaces of
$V$ (see Section \ref{sec:basic}). The majority of computation time in
the basic method is spent computing these $p_{\alpha\beta}$, so this
is where we will focus our optimisation efforts.

We can rewrite Serre's formula as:
\[
p_{\alpha\beta} = \frac{n}{|G|} \sum_{t \in G}
(\rho_i^*(t))_{\alpha\beta}\rho(t)
\]
where $^*$ denotes the dual representation (see Section
\ref{sec:reptheory} for its definition).

We can regard $p_{\alpha\beta}$ as an $n \times n$ block of a matrix
$p$ and notice that $p$ is exactly in the form of a tensor product of
$\rho_i^*$ and $\rho$. Define $\tau = \frac{n}{|G|}(\rho_i^* \otimes
\rho)$. Then $\tau(t)_{\alpha\beta} =
\frac{n}{|G|}((\rho_i^*(t))_{\alpha\beta} \rho(t)) =
\frac{n}{|G|}((\rho_i(t^{-1})_{\beta\alpha} \rho(t))$.

So in fact:
\begin{equation} \label{eqn:irred_prod}
p_{\alpha\beta} = \left(\sum_{t \in G} \tau(t)\right)_{\alpha\beta}
\end{equation}
where we take the $n \times n$ block in the $(\alpha, \beta)$
position, indexing over blocks.

It is not possible to use the same trick from Section \ref{sec:use_basis} to
speed up the computation of (\ref{eqn:irred_prod}). While we may have
bases for $C_{\rho_i^*}$ (which is spanned by an identity matrix) and
$C_\rho$ (provided by the user), this does not give us a way to
compute the basis for $C_{\rho_i^* \otimes \rho}$.

In general, it is \emph{not} the case that a basis for $C_f$ and $C_g$
will give us a basis for $C_{f \otimes g}$. An easy example is when $G
= S_3$, which has two degree one irreducibles: $\rho_{triv}$ and
$\rho_{sign}$, and a single degree two irreducible, $\rho$, the standard
representation.

A basis for $C_\rho$ is given by $I_2$ (see the argument in Section
\ref{sec:use_basis} for why). However, $\rho \otimes \rho =
\rho_{triv} \oplus \rho_{sign} \oplus \rho$, so $C_{\rho \otimes
  \rho}$ has dimension 3. We cannot immediately derive the basis for
$C_{\rho \otimes \rho}$ from the basis for $C_\rho$, since this is the
same thing as deriving the block structure of $\rho \otimes \rho$,
which might end up being significantly more complicated than that of
$\rho$. In this extreme case, we end up with a direct sum of
\emph{every} irreducible by tensoring a single irreducible with
itself.

We can, however, take advantage of the fact that the summation
(\ref{eqn:irred_prod}) has summands given by a group homomorphism.

Suppose we have a finite group $G$ and a monoid homomorphism $f : G
\to (R, \cdot)$, where $R$ is a ring with multiplicative monoid $(R,
\cdot)$.

We want to compute:

\begin{equation} \label{eqn:groupsum}
\sum_{g \in G} f(g)
\end{equation}

Given a subgroup $H \leq G$, we can compute the right coset
representatives of $H$: $r_1 = 1, r_2, \ldots, r_{|G:H|}$.

We can now rewrite the summation as:

\begin{equation} \label{eqn:groupsumtrick}
\sum_i \left(\sum_{h \in H} f(h)\right) f(r_i)
\end{equation}

This requires only $|H| + |G:H|$ image calculations, $|H| + |G:H|$
additions in $R$, and $|G:H|$ multiplications in $R$. If $H$ is either
trivial or $G$, this sum is the same as (\ref{eqn:groupsum}). To get
an improvement over naively summing, we want to choose a subgroup $H$
with ``medium-sized'' index, meaning that $|H| + |G:H| \leq |G|$.

Notice that we can now recurse and apply the same method to compute
the sum over $H$. If we were only choosing one $H$ to speed up the
sum, we would want this index to be ``medium-sized''. However, since
we can (by recursing) choose an arbitrarily long chain of subgroups
and we want to sum as few elements as possible, we want to split the
group as many times as possible. So we need to find a chain of
subgroups $G_1 \leq \ldots \leq G_n = G$ such that $|G_{i+1} : G_i|$
is minimal (and not 1) at each stage.

If $G = S_n$, it is easy to find a chain with fairly small indices.
Consider $G$ to act on $\{1, \ldots, n\}$ and use the chain $\{e\}
\leq S_2 \leq \ldots \leq S_n = G$, where we embed $S_i
\hookrightarrow S_n$ by considering $S_i$ to be the pointwise
stabiliser of $\{i+1, \ldots, n\}$. This gives us an improvement since
$|S_{i+1} : S_i| = i+1$, so there is a significant saving compared to
summing over the group directly. This idea generalises.

Let $G$ be a group acting on a set $\Omega$. A \emph{base} for $G$ is
a sequence $B = (b_1, \ldots, b_m)$ of points in $\Omega$ such that
the only $g \in G$ stabilising all $b_i$ is the identity.

The \emph{basic stabiliser chain} for $B$ is a chain of the form:

\[
\{1\} = G_m \leq \ldots \leq G_0 = G
\]

where $G_i$ is the group of the elements of $G$ stabilising the set
$\{b_1, \ldots, b_i\}$.

A \emph{strong generating set} for $G$ is a subset $S \subseteq G$
such that $S$ contains generators for all $G_i$.

GAP already has an implementation of the Schrier-Sims algorithm
\cite{sims_1970} \cite{knuth_1992} to calculate bases and strong
generating sets (BSGS), so we need not focus on the algorithmic
details of finding a BSGS.

Suppose we have a basic stabiliser chain for $G$, $G_m \leq \ldots
\leq G_0$. First, we compute the sum for $G_m$, which is
trivial. Then, we compute the sum for $G_{m-1}$ using the method
described earlier. We continue in this way until we reach $G$.

Although we are required to have a group with a natural action (i.e. a
permutation group) to use the BSGS algorithm to find a chain, the
following group summation algorithm is in no way specific to
permutation groups. Any chain of subgroups in any type of group
(permutation, matrix, finitely presented, and so on) will work: we
just do not describe methods of finding chains for these other types
of group.

\begin{algorithm}[h]
\begin{algorithmic}
\Function{sum group}{$f : G \to R$}
  \State let $\{1_G\} = G_m \leq \ldots \leq G_0 = G$ be a basic stabiliser chain for $G$
  \State $\sigma \gets f(1_G)$
  \For{$i \in \{m-1, \ldots, 0\}$}
    \State let $r_j$ be a transversal of representatives of $G_i \backslash G$
    \State $\sigma \gets \sum_j \sigma f(r_j)$
  \EndFor
  \State \Return $\sigma$
\EndFunction
\end{algorithmic}
\end{algorithm}

In general, the complexity of this algorithm depends greatly on the
basic stabiliser chain, which depends on the structure of the
group. To illustrate the speedup, we can consider the case when $G =
S_n$ and use the chain $\{e\} \leq S_2 \leq \ldots \leq S_n = G$. We
will measure the complexity in terms of ring operations, meaning
additions and subtractions in $R$.

Naively summing over the group requires $\O(n!)$ additions in $R$.

If we ignore the cost of computing coset representatives (which is
reasonable, since the cosets of $S_{m-1}$ in $S_m$ are known), to move
from the sum over $S_{m-1}$ to the sum over $S_m$, we compute:
\[
\sum_{i=1}^m \left(\sum_{h \in S_{m-1}} f(h)\right) f(r_i)
\]
where the $r_i$ are the $m$ coset representatives of $S_{m-1}$ in
$S_m$. Since the inner sum is already known, this takes $m$
multiplications and $m-1$ additions in $R$.

We do this for $m=2,3,\ldots,n$, so the total number of additions is
$\sum_{m=2}^n m-1$ which is $\O(n^2)$, and the number of
multiplications is similarly $\O(n^2)$. Even considering that $R$ is
usually a matrix ring, where multiplication is almost cubic in the
number of rows and columns, this is still asymptotically much better
than naively summing. In practice, we can sum $S_{10}$ in under a
second with the BSGS method, while it takes several minutes with the
naive method.

We can use this algorithm to compute the projections $p_{\alpha\beta}$
in Section \ref{sec:irr_decomp}.

Note that we cannot directly use this algorithm to compute the
projections to the canonical summands, since the formula $\sum_{t \in
  G} \overline{\chi_i(t)} \rho(t)$ does not have the property that the
summands are given by a group homomorphism.

However, from (\ref{thm:prop8}), we have that $p_i = \sum_{\alpha =
  1}^{n_i} p_{\alpha\alpha}$, so we can proceed by calculating the
$p_{\alpha\beta}$ first, then using these to calculate the $p_i$
projections onto the canonical summand and the decomposition of $V_i$.

\subsubsection{Summary}

To summarise, we apply the following optimisations:

\begin{enumerate}
  \item Determine which irreducibles we can ignore using the character inner
    product.
  \item \begin{enumerate}
  \item If the matrix of blocks $(p_{\alpha\beta})$ fits into memory,
    compute it using the group sum trick (\ref{eqn:groupsumtrick}),
    then use (\ref{thm:prop8}) to compute the projections $p_i$.
  \item If we are given an orthonormal basis for the centraliser and
    $\rho$ is unitary, compute the $p_i$ using the class sum trick
    (\ref{eqn:classsumtrick}).
  \item If $\rho$ is a permutation representation, compute an
    orthonormal basis for $C_\rho$ using orbitals, then go back to the
    previous case to compute $p_i$.
  \end{enumerate}
  \item Use the group sum trick to compute $(p_{\alpha\beta})$ if not
    already done and use these to break $V_i$ into irreducibles.
\end{enumerate}

The effect of these optimisations can be observed in the benchmarks in
Section \ref{sec:performance}.

The optimised version of the algorithm (which attempts to apply each
of the optimisations above) is the default when the functions {\tt
  IrreducibleDecomposition} and {\tt CanonicalDecomposition} are called.

\section{Our algorithm} \label{sec:mymethod}

We are trying to find the decomposition of a representation $\rho : G
\to M_n(\mathbb{C})$ into irreducibles, but given the irreducible
characters, we can immediately construct a representation $\tau : G
\to M_n(\mathbb{C})$ that is isomorphic to $\rho$ \emph{and} is block
diagonalised according to the irreducible decomposition.

In this section, we present an algorithm to find this $\tau$ and use
it to compute the irreducible decomposition of $V$ in the original
basis.

The GAP functions described in this section must be called with the
option {\tt decomp\_method := "alternate"} to use the algorithms
described here, the default methods are those described in Section
\ref{sec:serre}. For example, to call {\tt IrreducibleDecomposition}
in such a way as to use these methods, a user would call {\tt
  IrreducibleDecomposition(rho : decomp\_method := "alternate")}.

\subsection{Finding the block diagonalisation}

Two representations of a finite $G$ over $\mathbb{C}$ are isomorphic
if and only if they have the same character. Let the list of
irreducible characters of $G$ by given by $\chi_1, \ldots,
\chi_N$. These form an orthonormal basis for the $\mathbb{C}$-vector
space of characters, with respect to the inner product $\langle
\chi_i, \chi_j \rangle = \frac{1}{|G|} \sum_{g \in G} \chi_i(g)
\overline{\chi_j(g)}$.

We can thus determine the irreducible decomposition of $\rho$ (up to
isomorphism as representations) by calculating the multiplicities
$m_i$ in the expression $\chi_\rho = \sum_{i=1}^N\;m_i\,\chi_i$:
$\langle \chi_\rho, \chi_i \rangle = m_i$.

This means $\tau$, the block diagonalised representation isomorphic to
$\rho$, is given by $\tau = \bigoplus\limits_{i = 1}^N \bigoplus\limits_{j = 1}^{m_i}
\rho_i$, where $\rho_i$ is the irreducible corresponding to the character
$\chi_i$.

The GAP function {\tt BlockDiagonalRepresentation(rho)} returns this
$\tau$, given $\rho$.

\subsection{Finding the intertwining operator} \label{sec:inter}

We have constructed a block diagonal $\tau$, but we would now like to
know what the irreducible $G$-invariant spaces are, in our original
basis.

In the basis $\tau$ is written in, the irreducibles are spanned by certain
subsets of $\{e_1, \ldots, e_n\}$. To translate these vectors into the
old basis, we need to calculate the linear map $A : \mathbb{C}^n \to
\mathbb{C}^n$ with the property that:

\begin{equation} \label{eqn:iso}
A^{-1} \tau(g) A = \rho(g)\;\mbox{for all }g \in G
\end{equation}

$A$ is an \emph{intertwining operator} or isomorphism between the
representations.

We can find the intertwining operator by observing that the action of
$G$ on matrices given by:
\[
g \mapsto (A \mapsto \tau(g)A\rho(g^{-1}))
\]

is in fact a linear action, so is a representation of $G$, call it
$\alpha$.

$\alpha$ is a map from $G$ into the $n^2$ dimensional (where $n$ is
the degree of $\rho$ and $\tau$) $\mathbb{C}$-vector space
$M_n(\mathbb{C})$. Using methods from Section \ref{sec:serre}, we can
compute $\alpha$ and find the canonical summand $V_{triv}$
corresponding to the trivial representation $g \mapsto
\left(1\right)$.

If $A \in V_{triv}$, then since $G$ acts as the identity on this
subspace, $\alpha(g)A = A$. This means for all $g \in G$, $A =
\tau(g)A\rho(g^{-1})$, which is exactly the property required in
(\ref{eqn:iso}) for $A$ to be the intertwining operator.

The running time of this method depends on the running time of finding the
trivial canonical summand of a representation of $G$, which is done
using Serre's formula for the projection $p : V \to V_{triv}$:

\begin{equation} \label{eqn:trivproj}
  p = \sum_{g \in G} \alpha(g)
\end{equation}

Notice that this summation has summands given by a homomorphism, so we
can use the method of building the sum from a BSGS, as in Section
\ref{sec:groupsum}.

We then pick a random point $B$ in the domain of $p$, calculate the
image point $pB$, and rely on the fact that $pB$ will almost always be
invertible.

To be precise, the map $B \mapsto \det(pB)$ is polynomial in the
entries of $B$. We know that it is not identically zero since we know
an isomorphism between $\rho$ and $\tau$ exists, by
construction. Nonzero polynomials are zero on hypersurfaces with
dimension strictly less than the dimension of the ambient space. In
particular, this means the set of $B$ with $pB$ singular has measure
zero, since hypersurfaces have measure zero. In the pure mathematical
sense, this means that picking $B$ from a uniform distribution on a
ball with nonzero radius has probability zero of being singular. In
practice, with a computer, the probability is low enough that a single
try will almost always\footnote{Here, I do not mean ``almost always''
  in the precise mathematical sense, I mean that the probability is
  fairly close to 1} be enough to find an invertible $pB$.

We can directly prove, without relying on Serre's proof, that $A := pB$
satisfies the intertwining operator property. Let $g_0 \in G$. Then:

\begin{equation} \label{eqn:isoproof}
\begin{aligned}
 \tau(g_0) A &= \tau(g_0) \sum_{g \in G} \tau(g) B \rho(g^{-1}) \\
  &= \sum_{g \in G} \tau(g_0g) B \rho(g^{-1}) \\
  &= \sum_{g \in G} \tau(g) B \rho(g^{-1}g_0) & \mbox{(by
    relabelling)} \\
  &= \sum_{g \in G} \tau(g) B \rho(g^{-1}) \rho(g_0) \\
  &= A \rho(g_0)
\end{aligned}
\end{equation}

While we do know what $\alpha$ is as a linear map, we need to know
what it is as a matrix in order to actually compute the sum.
\vspace{1em}
\begin{proposition}
$\alpha = \tau \otimes \rho^*$, where $\rho^*(g) = \rho(g^{-1})^T$
  denotes the dual representation (defined in Section
  \ref{sec:reptheory}).

  (We consider $\alpha(g)$ to act on $n \times n$ matrices by reading
  off each row in sequence, one after the other, to obtain a vector in
  $\mathbb{C}^{n^2}$)
\end{proposition}

\begin{proof}
Recall that $e_i \otimes e_j = e_i e_j^T = E_{ij}$, and that the
$E_{ij}$ form a basis for the space of matrices
$M_n(\mathbb{C})$. Then:
\begin{align*}
(\tau \otimes \rho^*)(g)E_{ij}
&= \tau(g)e_i \,\otimes\, \rho^*(g)e_j \\
&= \tau(g)e_i (\rho^*(g)e_j)^T \\
&= \tau(g) e_i e_j^T \rho(g^{-1}) \\
&= \tau(g) E_{ij} \rho(g^{-1}) \\
&= \alpha(g)E_{ij} \\
\end{align*}
\end{proof}

If we directly calculate the images $\alpha(g)$ using the Kronecker
product of two $n \times n$ matrices, this will incur $\mbox{O}(n^4)$
extra space. This is a problem if the degree of the representation is
large, since even a single matrix $\alpha(g)$ will not fit in memory.

\subsection{Reducing the degree} \label{sec:myparallel}

One method of reducing the degrees we need to compute with is to
proceed in two stages. First, split $V$ into canonical summands using
the method due to Serre described in Section
\ref{sec:canonical}. Second, apply the method from Section \ref{sec:inter} on
the (hopefully much smaller degree) canonical summands. In the worst
case, this will provide no benefit, since it is possible that there is
a single canonical summand - the whole space. Generally,
representations do consist of more than one isomorphism class of
irreducible, so this optimisation is worthwhile, especially considering that
the calculation for each summand can now happen in parallel.

Given a linear map $A : V \to V$, we can restrict it to a subspace $W$
with basis $\{ w_j \}_{j=1}^{\dim W}$ by computing the matrix
$(A|_W)_{ij} = A(w_j)_i$. We can restrict a representation $\rho$ by
computing $\rho(g)|_W$ for each generator $g$ of $G$.

A key property used in Section \ref{sec:use_basis} to compute the centraliser
basis is that once we have block diagonalised, blocks corresponding to
isomorphic irreducibles will have the same matrix coefficients. Restricting
to a subspace preserves this property, as long as we are careful to
use the correct bases.

\subsection{Memory-constrained methods} \label{sec:lowmem}

In some cases, the Kronecker products will still be too large to fit
in memory. In this case, we must find a way to represent $\alpha$
without using Kronecker products, then find an element of $V_{triv}$.

The key space optimisation is to represent a tensor product of
matrices as a pair of matrices, without actually calculating the
Kronecker product. This takes advantage of the (multiplicative) monoid
homomorphism $\phi : M_n(\mathbb{C}) \times M_n(\mathbb{C}) \to
M_n(\mathbb{C}) \otimes M_n(\mathbb{C})$ given by $\phi(A, B) = A
\otimes B$. This is a homomorphism since $(A \otimes B)(C \otimes D) =
(AC \otimes BD)$.

Given $g \in G$, we can represent $\alpha(g)$ as $(\tau(g),
\rho^*(g))$. The action of $\alpha$ on $n \times n$ matrices is then
given by: $(\tau(g), \rho^*(g))E_{ij} = \tau(g)e_i \otimes
\rho^*(g)e_j = \alpha(g)E_{ij}$ (extending to all matrices using
linearity). Using the pair representation, we are never required to
explicitly compute and store the matrix representing $\alpha(g)$, thus
getting rid of the problematic $\mbox{O}(n^4)$ space usage.

The pair representation does not allow us to sum tensor products, so
we cannot use the same method as before. Let $A \in M_n(\mathbb{C})$
and let $G \cdot A$ be its orbit under the action of $G$ given
by $\alpha$. If $v = \sum_{x \in G \cdot A} x$, then $v \in
V_{triv}$, since the action of $G$ fixes the set $G \cdot A$ and
is linear, so fixes their sum. Define the map $f : M_n(\mathbb{C}) \to
M_n(\mathbb{C})$ by:

\[
f(A) = \sum_{x \in G \cdot A} x
\]

then by a similar argument as for (\ref{eqn:trivproj}), this map
almost always gives, in a precise sense, an invertible matrix
$f(A)$. This matrix will satisfy the intertwining operator property.

The algorithm used to compute the orbit sum $f(A)$, where the action
of $G$ is given by $\rho \otimes \tau$ is as follows:

\begin{algorithm}[H]
\begin{algorithmic}
\Function{orbit sum}{$A$, $\rho \otimes \tau$}
  \State let $\{g_1, \ldots, g_n\}$ be a generating set for $G$
  \State $\sigma \gets 0$
  \State $\Delta \gets [A]$ \Comment{singleton list}
  \State $i \gets 1$
  \While{$i \neq |\Delta|$}
    \State $v \gets \Delta[i]$
    \State $\sigma \gets \sigma + v$
    \For{$1 \leq i \leq n$}
      \If{$\rho(g_i) \otimes \tau(g_i) \cdot v \notin \Delta$}
        \State append $\rho(g_i) \otimes \tau(g_i) \cdot v$ to $\Delta$
      \EndIf
    \EndFor
    \State $i \gets i + 1$
  \EndWhile
  \State \Return $\sigma$
\EndFunction
\end{algorithmic}
\end{algorithm}

A downside to this method is that it is possible that the list used to
hold orbit elements could need to grow to a large size, possibly
$|G|$. If this orbit cannot fit into memory, this orbit summing method
will not work, and we must resort to naively summing over $G$ to
calculate an image of the projection to $V_{triv}$:

\[
\sum_{g \in G} \alpha(g)A
\]

for some randomly chosen $A$. This requires a number of matrix
additions and multiplications linear in the size of $G$ but only
requires us to store a small constant number of $n \times n$ matrices
(an accumulator and each summand, one by one), so is not very
memory-intensive.

\section{Testing and benchmarking} \label{sec:performance}

The benchmarks and tests in this section were run on a laptop with an
Intel Core i7-4720HQ CPU running at 2.60GHz, with 12 GB of memory. The
code was run on Debian GNU/Linux, with the 4.15.0 kernel and GAP
version 4.10.1.

All times reported are wall-clock times, measured using GAP's {\tt
  NanosecondsSinceEpoch} function. Each benchmark was run 3 times and
the result averaged to produce each data point.

\subsection{Generating test cases} \label{sec:test}

When converting our pseudocode and theorems into runnable code, it is
helpful to have a wide range of test cases so that we can be fairly
sure our program is correct.

One approach is to think of a wide range of examples and write tests
that check the correctness of the algorithm on those specific
cases. While this is better than having no tests, it is unlikely that
we will be able to come up with a complete set of examples.

The approach we took was to randomly generate test cases and check
properties. This is known as \emph{property-based testing} and is
inspired largely by QuickCheck\footnote{\url{
 http://hackage.haskell.org/package/QuickCheck}}, a Haskell library
for property-based testing.

Our GAP package has a function, {\tt RandomRepresentation} that
generates a representation by the following procedure:

\begin{itemize}
\item Randomly select a group $G$ from GAP's {\tt SmallGroup} library.
\item Compute the list of irreducible representations of $G$ using
  {\tt IrreducibleRepresentationsDixon}
\item Choose a small number of these irreducibles randomly and
  directly sum them using our function {\tt
    DirectSumOfRepresentations}.
\item Conjugate by a random invertible matrix, generated by {\tt
  RandomInvertibleMat}.
\end{itemize}

While computing this random representation, since we constructed it
from irreducibles, we know the block structure and centraliser. This
knowledge is what we use to check the correctness of the decomposition
of the representation after the computation.

For each function in the package, our test suite generates several
test cases and checks the correctness of the result.

\subsection{Performance comparisons}

There are two types of benchmarks that were conducted: benchmarks on a
large set of random small groups (generated using {\tt
  RandomRepresentation}), and benchmarks on some known examples.

Since there are too many combinations of options and flags to test
them all, we focus on several cases that demonstrate the effect of our
optimisations.

First, we will benchmark the method described in Section \ref{sec:mymethod},
comparing the different ways to find intertwining operators
(\ref{sec:inter}). Recall that to find an isomorphism $\rho \cong
\tau$, we need to find a vector in the fixed space $V_{triv}$ of
$\alpha = \rho \otimes \tau^*$, there are several ways to do this:

\begin{enumerate}
  \item Naively summing over $G$ (referred to as the ``naive'' method)
  \item Only summing an orbit of $\alpha$ (referred to as the ``orbit
    sum'' method)
  \item Computing the projection to $V_{triv}$ using the fast group
    sum trick (\ref{sec:groupsum}) and the Kronecker product of $\rho$
    and $\tau^*$ (referred to as the ``Kronecker'' method).
\end{enumerate}

Consider the representation of $S_n$ defined by $\rho_n : S_n \to
\mathbb{C}^n$, $\rho(\sigma)_{ij} = 1$ if $j = \sigma(i)$ and 0
otherwise. This is clearly a faithful and unitary representation. It
splits into a direct sum of two irreducible representations: the
trivial summand spanned by the all-one vector, and the orthogonal
complement (by the $G$-invariant inner product). The nontrivial irreducible
is known as the standard representation of $S_n$.

\begin{figure}[H]
  \input{graphs/mymethod_symmetric_degree.pgf}
\end{figure}

It is immediately apparent that the orbit sum method is the
worst. This is because a random vector will likely have $n$ distinct
entries, which means its orbit under $S_n$ will have $n!$ elements: we
get every permutation of the entries. Thus summing over the orbit is
not an optimisation in this case: our running time is still linear in
$|G|$.

The next worst method is the naive method. This has the same
asymptotic behaviour as the orbit sum method, but with the overhead of
keeping track of the orbit eliminated. It is thus faster, but is still
linear in $|G|$.

We see a significant improvement when we use the Kronecker
method. Note that $\rho_n$ has degree $n$, and this is very small
compared to the size of the group, $n!$. This means that the matrices
we are working with are not very large, so the Kronecker products are
also not too large.

We can observe a similar pattern in two of the methods used to
implement Serre's algorithm:

\begin{enumerate}
  \item Directly computing the projections $p_{\alpha\alpha}$ by
    summing over $G$ (referred to as the ``naive'' method)
  \item Computing $p$ as the projection to the trivial subspace of
    $\rho_i^* \otimes \rho$, using the group sum trick from
    Section \ref{sec:groupsum} (referred to as the ``Kronecker'' method).
\end{enumerate}

\begin{figure}[H]
  \centering
  \input{graphs/serre_symmetric_degree.pgf}
\end{figure}

We observe that the Kronecker is superior to the naive method in this
case, where the degree of the representation is small and the group is
large, as before.

Next, we consider the regular representations of $\mathbb{Z}_n$, the
group of integers modulo $n$. $\mathbb{Z}_n$ is an abelian group, so
all irreducibles are of dimension $1$, and appear in the regular
representation $V$ with multiplicity $1$. For the purposes of
demonstration, we will consider the tensor product $V \otimes V$, as
this has a larger degree, $n^2$.

\begin{figure}[H]
  \input{graphs/mymethod_cyclic_size.pgf}
\end{figure}

We observe that the Kronecker method is still the most efficient. This
is despite the small group size, which improves the running time of the
other methods, which sum over the group. Our conclusion is that the
Kronecker method is a fairly efficient method. We will compare it to
Serre's method later on.

In our implementation of Serre's method, we see the naive method
remain competitive. The cyclic group size is small enough that naively
summing over the group is faster than the the method using Kronecker
products and the group sum trick.

\begin{figure}[H]
  \input{graphs/serre_cyclic_size.pgf}
\end{figure}

This is demonstrates a special case where the naive method is fastest:
when the group is extremely small. In most other cases, the Kronecker
method will be better, as we saw when decomposing representations of
$S_n$.

We also want to test the dependency of the running times on the degree of
the representations, while keeping the size of the group constant. To
demonstrate, we consider the representations $\tau_n : \mathbb{Z}_3
\to \otimes^n V$ where $V$ is the 3 dimensional regular representation
of $\mathbb{Z}_3$.

\begin{figure}[H]
  \input{graphs/mymethod_tensor_degree.pgf}
\end{figure}

Again, we see that the Kronecker method is the fastest.

Next, we will explore the effect of parallelisation on our method (see
Section \ref{sec:myparallel} for more details) and Serre's method. We do not
expect to see a large improvement in cases where a representation is
made up of few irreducibles, so we will benchmark the methods on the tensor
$V \otimes V$ of the regular representation $V$ of $\mathbb{Z}_n$
described earlier, since these representations have many irreducibles.

To parallelise our method, we iterate over the list of irreducibles in
parallel: for each, we project to the canonical summand using Serre's
formula and run the original non-parallel method on the canonical
summand.

\begin{figure}[H]
  \input{graphs/mymethod_cyclic_kronecker_size.pgf}
\end{figure}

\begin{figure}[H]
  \input{graphs/mymethod_cyclic_naive_size.pgf}
\end{figure}

For both methods, we see a large improvement in some cases. The reason
for this is that there are $n$ irreducibles of $\mathbb{Z}_n$ over
$\mathbb{C}$, one for each $n$th root of unity. This means there are
$n$ canonical summands, so we could split the work between $n$
processors effectively. The laptop used to run these benchmark only
has 4 physical cores, so we do not observe the speedup we would in an
ideal scenario, where there is one CPU for each canonical summand.

Similarly to our method, we parallelise Serre's method by splitting
into canonical summands and running the original method on each
summand.

\begin{figure}[H]
  \input{graphs/serre_cyclic_kronecker_size.pgf}
\end{figure}

\begin{figure}[H]
  \input{graphs/serre_cyclic_naive_size.pgf}
\end{figure}

For these examples, the overhead of parallelising is too great and we
do not see an improvement in running time for these small examples.

It is worth precisely how the parallelisation works, since
this reveals why there is no speedup. We use the GAP function {\tt
  ParListByFork} to process lists in parallel. This function forks
child processes, performs the computation, then communicates the
results over a pipe to the parent, which builds the result list from
each child's result. The time this takes is highly dependent on the
performance characteristics of the OS {\tt fork} system call, IPC
(interprocess communication) performance, and various GAP
serialisation functions. Ideally, we would parallelise using more
finely-grained concurrency primitives to avoid the need for any IPC,
but none are available in GAP. When profiling, we found that a large
amount of running time is spent in functions such as {\tt IO\_Pickle} and
{\tt IO\_Unpickle}, which are serialisation functions, used to
transmit results between processes.

Finally, we compare our method to Serre's method. For our method, we
will use the paralellised Kronecker product method, since this was the
fastest in all cases. For Serre's method, we will benchmark both the
naive and Kronecker product methods. This is because the naive method
is faster for small groups, while the Kronecker method is faster for
large groups.

We benchmarked the algorithms for a random selection of $50$ test
cases, with group sizes from $1$ to $200$ and the degree being less
than $10$.

\begin{figure}[H]
  \input{graphs/vs_random_size.pgf}
\end{figure}

Serre's method with Kronecker products performed extremely poorly in
most cases, it was never the best method for any representation
tested.

Our method with Kronecker products was the fastest method 35 times,
while Serre's method with naive summing was the fastest 15 times.

Next, we benchmarked the three methods on the defining representation
of the symmetric group.

\begin{figure}[H]
  \input{graphs/vs_symmetric_size.pgf}
\end{figure}

This time, we see that Serre's naive method is the worst - this is due
to the large size of the symmetric groups. Our method is again the
best.

Lastly, for an example where our method is not the best, we
benchmarked the tensor product of the regular representation of
$\mathbb{Z}_n$.

\begin{figure}[H]
  \input{graphs/vs_cyclic_size.pgf}
\end{figure}

Here, Serre's naive method is the best: its running time is dominated by
the size of the group which is small. Our method has a running time
dominated by a polynomial in the degrees of the representations, which
grow quadratically ($V \otimes V$ has dimension $n^2$) with the size
of the group.

We can conclude that for small groups, Serre's naive algorithm has
good performance. For larger groups, our Kronecker method usually has
good performance. In general, it is hard to say anything more concrete
than this, since performance depends on too many factors - it is not
just the size of the group that affects running time. We have different
numbers of conjugacy classes, special cases in GAP algorithms for
certain groups (e.g. integers, symmetric groups), different internal
representations of certain groups (e.g. polycyclic, presentations) and
so on.

\section{An application: Bounding the crossing number of $K_{m,n}$} \label{sec:crossing}

Being able to efficiently decompose representations of finite groups
leads to a wide variety of possible applications. In semidefinite
programming, it is common in several applications that a program has
symmetries which can be expressed as a group action and used to
greatly reduce the dimension of the problem. In this section, we
present a method due to de Klerk et al. \cite{deKlerk2007} for
reducing the dimension of such semidefinite programs. We also apply
this method to the problem of computing a lower bound for the crossing
number of the complete bipartite graph.

While the reduction in the dimension of the semidefinite program is
done by us in the same way as de Klerk et al. \cite{deKlerk2007}, the
key difference is that we can now optimally block diagonalise the
representation of the action and thus the centraliser. This is an
improvement over the original method by de Klerk et al.
\cite{deKlerk2004}, in which no general method of optimally block
diagonalising is given.

\subsection{Motivation}

A \emph{complete bipartite graph} $K_{m,n}$ is a graph that can be
partitioned into two sets. One with $m$ vertices, one with $n$
vertices, and each vertex in one set connected with an edge to each
vertex in the other set, with no other edges.

In 1954, Zarankiewicz published a proof that the crossing number of
$K_{m,n}$, $\mbox{cr}(K_{m,n}) = \lfloor \frac14(m-1)^2 \rfloor \lfloor
\frac14 (n-1)^2 \rfloor$ \cite{Zar1955}. His argument contained an
error, meaning the proof was only valid for $K_{3,n}$. Other proofs
have been published, proving the conjecture for $\min(m,n) \leq 6$
\cite{Kleitman} and $m \in \{7,8\}$ and $7 \leq n \leq 10$
\cite{Woodall}.

The truth of the conjecture is not known in general, but bounds on
$\mbox{cr}(K_{m,n})$ are known. Let $Z(m,n) = \lfloor \frac14(m-1)^2
\rfloor \lfloor \frac14 (n-1)^2 \rfloor$.

\begin{center}
\begin{figure}[h]
\begin{center}
\begin{tikzpicture}
  \Vertex[x=0,y=2]{A}
  \Vertex[x=0,y=0]{B}
  \Vertex[x=0,y=-2]{C}

  \Vertex[x=4,y=1]{1}
  \Vertex[x=4,y=-1]{2}

  \Edge(A)(1)
  \Edge(A)(2)
  \Edge(B)(1)
  \Edge(B)(2)
  \Edge(C)(1)
  \Edge(C)(2)
\end{tikzpicture}
\hspace{1em}
\begin{tikzpicture}
\Vertex[x=0,y=2]{A}
\Vertex[x=0,y=1]{B}
\Vertex[x=0,y=-1]{C}

\Vertex[x=1,y=0]{1}
\Vertex[x=-1,y=0]{2}

\Edge(A)(1)
\Edge(A)(2)
\Edge(B)(1)
\Edge(B)(2)
\Edge(C)(1)
\Edge(C)(2)

\end{tikzpicture}
\end{center}
\caption{Two drawings of $K_{3,2}$. The second has $Z(3,2) = 0$ crossings.}
\end{figure}
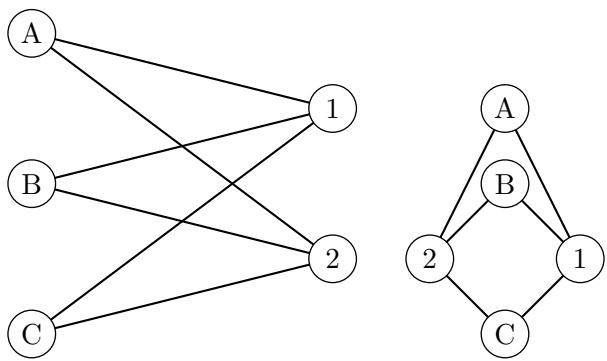
\end{center}

De Klerk et al. \cite{deKlerk2007} showed that:

\[
\lim_{n \to \infty} \frac{\mbox{cr}(K_{m,n})}{Z(m,n)} \geq 0.8594 \frac{m}{m-1}
\]

This is the result that we will reproduce in this section. Our
improvement is that, due to a reduction in the sizes of blocks of the
matrices we compute with, the computation is more efficient.

\subsection{Reducing semidefinite programs}

A more detailed presentation of this method, with full proofs of
correctness, can be found in the paper by de Klerk, Pasechnik and
Schrijver \cite{deKlerk2007}.

Suppose we are given a semidefinite program:

\newcommand{\tr}{\mbox{tr}}

\begin{equation} \label{eqn:unreduced}
\mbox{min}\{ \tr(CX) \mid X \;\mbox{positive semidefinite}, X \geq 0, \tr(A_j
X) = b_j \;\mbox{for}\; j = 1, \ldots, m \}
\end{equation}

Suppose additionally that we have a finite group $G$ acting on a
finite set $Z$ with $\rho$ being the permutation representation of this
group action, i.e. $\rho(g)$ is the matrix with rows and columns
indexed by $Z$ with $\rho(g)_{xy} = 1$ iff $gx = y$ and $0$
otherwise. We also require that the $C$ and $A_j$ matrices in
(\ref{eqn:unreduced}) are elements of the centraliser ring of this
representation, meaning they commute with all $\rho(g)$.

Let $\{E_1, \ldots, E_d\}$ be the orbital matrices for $G$ acting on
$Z$, and for each $i$, $B_i = \frac{E_i}{\Vert E_i \Vert}$, where the
inner product is given by $\langle X, Y \rangle = \tr(XY^*)$. Define
$L_k$ such that $B_k B_j = \sum_i (L_k)_{ij} B_i$. This is possible
since the centraliser ring (spanned by the $E_i$) is closed under
multiplication.

A key result of de Klerk et al. \cite{deKlerk2007} is that the program
(\ref{eqn:unreduced}) is equivalent to the following program:

\begin{align*}
\mbox{min} \left\{ \sum_{i=1}^d \tr(CB_i)x_i \middle\vert \right.&\sum_{i=1}^d x_i L_i
\;\mbox{positive semidefinite}, x_i \geq 0 \;\mbox{for}\; i = 1,
\ldots, d, \\
&\left. \sum_{i=1}^d \tr(A_j B_i)x_i = b_j \;\mbox{for}\; j = 1, \ldots, m \right\}
\end{align*}

The number of variables $x_i$ is $d$, the dimension of the centraliser
ring: this is in most cases a significant reduction from the number of
variables that need to be considered in (\ref{eqn:unreduced}).

Another reduction in complexity comes from the fact that we only
condition on positive semidefiniteness of matrices of the size of the
$L_i$, which are $d \times d$ matrices. Compared with
(\ref{eqn:unreduced}), which constrains on the positive
semidefiniteness of $X$, a $Z \times Z$ matrix, this can be a large
saving. As an example, in this section we will consider a case where
$|Z| = 6! = 720$ while $d = 78$.

De Klerk et al. \cite{deKlerk2007} note further that we can always
find a symmetric matrix exhibiting the optimal value. Recall that each
orbital $\Delta_i$ has a paired orbital $\Delta_{i^*}$ with all of the
pairs swapped. This corresponds to the adjacency matrices being
transpose: $E_i = E_{i^*}^T$. A consequence is that we can perform the
restriction to symmetric matrices by adding the constraints $x_i =
x_{i^*}$ for each $i$.

We can apply this to crossing numbers as follows.

Let $G = S_m \times S_2$ act on $Z_m$, the set of $m$-cycles in $S_m$
by: $(\sigma, e) \cdot \rho = \sigma \rho \sigma^{-1}$ and $(e, \tau)
\cdot \rho = \rho^{\text{sign}(\tau)}$.

Define a matrix $C$ in $\mathbb{R}^{Z_m \times Z_m}$ by $C_{\sigma,
  \tau} = $ the minimum number of adjacent interchanges required to
transform $\sigma$ to $\tau^{-1}$. An adjacent interchange is a swap
of elements that are next to each other when the cycle is written
down. For example, $(1423)$ and $(4123)$ differ by an adjacent
interchange. Then $C$ is in the centraliser of the $G$-action. See
\cite{deKlerk2004} for more details on the matrix $C$.

Then define a constant $\alpha_m$ by:
\[
\alpha_m := \mbox{min}\{\mbox{tr}(CX) \mid X \in \mathbb{R}^{Z_m
  \times Z_m}, X \mbox{positive semidefinite}, \mbox{tr}(JX) = 1\}
\]

where $J$ is the all-one matrix.

de Klerk et al. proved \cite{deKlerk2007} that:

\[
\mbox{cr}(K_{m,n}) \geq \frac{m(m-1)}{k(k-1)}\left(\frac12 n^2 \alpha_k -
\frac12 n \left\lceil \frac14 (k-1)^2 \right\rceil\right)
\]

for all $n$ and $k \leq m$, which implies:

\[
\lim_{n \to \infty} \frac{\mbox{cr}(K_{m,n})}{Z(m,n)} \geq
\frac{8\alpha_k m}{k(k-1)(m-1)}
\]

for all $n$ and $k \leq m$.

We proceeded by reducing the semidefinite program defining $\alpha_m$
using the method due to de Klerk et al. \cite{deKlerk2007}, then block
diagonalising the representation using our algorithm (using naive
summing) described in Section \ref{sec:lowmem}. We were forced into using this
method since the degrees of the representations are large: the
Kronecker product methods quickly ran out of memory.

We provided our algorithm with a complete list of irreducibles of $S_n$
using Young tableaux. This algorithm was not implemented in GAP, so we
used the implementation from Sage. The reason for using this algorithm
over GAP's generic Dixon algorithm is that computing with Young
tableaux is much faster and allows us to produce integer matrices:
sidestepping any issues involving cyclotomic fields (see
Section \ref{sec:gram} for why this is important).

\subsection{Results}

These results were computed on a laptop with an Intel Core i7-4720HQ
CPU running at 2.60GHz, with 12 GB of memory. The results for
$\alpha_m$ for $m \in \{5,7\}$ match approximately those obtained in
earlier papers \cite{deKlerk2004} \cite{deKlerk2007}. Values for
$\alpha_m$ for $m \geq 8$ were not computed here due to the time
required, but values for $m=8,9$ were computed by de Klerk et al.
\cite{deKlerk2007}.

The computation took less than a second for $m=5$, 10 seconds for
$m=6$ and 1.5 hours for $m=7$. There is much room for improvement on
these computation times, but optimising this problem was not the goal
of this project - we merely aim to prove that our algorithm works in a
real-world scenario.

\begin{table}[h]
\begin{center}
\begin{tabular}{|l|l|}
  \hline
  $m$ & $\approx \alpha_m$ \\
  \hline
  5 & 1.9472133720059 \\
  6 & 2.9519170848593 \\
  7 & 4.3693933617464 \\
  \hline
\end{tabular}
\end{center}
\caption{Results for $5 \leq m \leq 7$}
\end{table}

These were produced by running the script {\tt ./run\_crossing.sage m}
where {\tt m} is the value of $m$.

To demonstrate the efficiency of our block diagonalisation, we can
examine the decomposition of the $m=7$ representation. $m=7$ has
degree $720$ (it acts on the set of 7-cycles, of which there are
$6!$), and decomposes into blocks as follows:

\begin{table}[h]
\begin{center}
\begin{tabular}{|l|l|}
  \hline
  Block Size & Number of Blocks \\
  \hline
  1 & 2 \\
  14 & 8 \\
  15 & 6 \\
  20 & 2 \\
  21 & 6 \\
  35 & 10 \\
  \hline
\end{tabular}
\end{center}
\caption{Block sizes for $m=7$}
\end{table}

Notice that the largest block size is 35, considerably smaller than
720, meaning our method has been effective in reducing the sizes of
matrices we are required to compute with.

\section{Conclusions and future work} \label{sec:conclusion}
\subsection{Conclusions}

We have satisfied all of the requirements set out in
Section \ref{sec:introduction}. To summarise how we achieved this:

We provide two functions to decompose a representation: {\tt
  IrreducibleDecompositionCollected}, which returns a list of lists of
irreducibles. The irreducibles are collected into lists according to their
isomorphism class, so all spaces appearing in a list are isomorphic as
representations. {\tt IrreducibleDecomposition} is the flattened
version, a list of irreducibles. In this case, each irreducible is given as a
subspace of $V$, in the form a of a GAP vector space.

To compute the block diagonalisation of $\rho$, we provide the
function {\tt BlockDiagonalRepresentation}, which converts a
representation $\rho$ to a block diagonal one $\tau$. $\tau$ is given
as a GAP homomorphism from the group to a matrix group. The matrix we
conjugated by to get $\tau$ is given by {\tt
  BlockDiagonalBasisOfRepresentation} applied to $\rho$.

To compute a basis for the centraliser ring $C_\rho$, we provide the
function {\tt CentralizerOfRepresentation}, which gives the matrices
spanning $C_\rho$ as a vector space, written in the same basis as
$\tau$. We also provide {\tt CentralizerBlocksOfRepresentation}, which
gives the same matrices, but as lists of blocks rather than full
matrices.

Our property-based testing method is described in Section \ref{sec:test}. We
found that this method was superior to manually writing a complete set
of examples, since it caught more errors. Since our algorithm applies
to representations of any group, it is almost impossible to think of
examples covering every possible case: solvable/unsolvable,
abelian/nonabelian, cyclic/not cyclic, nilpotent/not nilpotent, and so
on. Through the random generation of large numbers of test cases, we
discovered edge cases where one of our assumptions broke down. A
specific example was encountered when we were testing the centraliser
trick (see Section \ref{sec:use_basis}). Originally we only tested on unitary
representations and did not realise our assumption that
representations were unitary - after implementing randomised testing,
this was quickly detected and resolved.

Lastly, we wrote an extensive GAP package manual, documenting all
functions, arguments, preconditions, return values and so on. We also
made sure to include comments on which algorithms are best in which
cases (derived from experiments conducted in Section
\ref{sec:performance}). The source code is also heavily commented,
hopefully allowing future contributions without much difficulty.

\subsection{Future Work}

While we have succeeded in fulfilling the requirements, there are
still more improvements and additions that could be made.

We could have written an implementation of Dixon's algorithm for
decomposing unitary representations, discussed in the appendix
\ref{sec:unitary}. This algorithm avoids the need for a complete list
of irreducibles to be provided beforehand, which is a property none of the
other algorithms have and warrants further exploration. Once we factor
in the computation time for the list of irreducibles, it is possible that we
would discover cases where Dixon's algorithm is the best. This task
was out of scope for this project, but could form the basis for future
contributions to GAP or Sage.

In terms of performance analysis, there is much more that could be
done. For example, we could have written tools similar to
KCachegrind\footnote{\url{
    http://kcachegrind.sourceforge.net/html/Home.html}} to produce
graphs and analyse the call graphs of our functions in more
detail. The analyses we did perform were done with the aid of GAP's
profiling tools, which are not as advanced as the tools included with
valgrind\footnote{\url{http://valgrind.org/}}, for example.

We did not profile and analyse memory usage, as this was not a primary
concern compared to running time. This is another area of improvement:
we did not heavily optimise for memory usage, instead opting to trade
memory for a better running time wherever possible. This is most
apparent in our usage of Kronecker products, which grow the size of
the matrices we are computing with from $n \times n$ to $n^2 \times
n^2$. This generally improves running time, but memory usage
suffers. We saw this in Section \ref{sec:crossing}, where the
Kronecker product methods were too memory intensive when $n=720$, and
could not be used. In some cases, we may have been able to use sparse
matrices to improve memory usage, since some representations have
matrices which have mostly zeroes as entries. One GAP package that
implements this is {\tt
  Gauss}\footnote{\url{https://www.gap-system.org/Packages/gauss.html}},
which implements a sparse matrix data structure that only keeps track
of the non-zero entries. We could also use sparse matrix data
structures from
SageMath\footnote{\url{http://doc.sagemath.org/html/en/reference/matrices/sage/matrix/matrix_sparse.html}}
or
SciPy\footnote{\url{https://docs.scipy.org/doc/scipy/reference/sparse.html}},
but this would require implementing parts of our algorithm in SageMath
to take advantage of the interoperation between GAP and Python
libraries.

\subsection{Project Availability}

All code associated with the project is currently available at
\url{https://gitlab.com/kaashif/decomp}. The GAP package {\tt
  RepnDecomp}, containing all functionality implemented as a part of
this project, will be submitted for review and will become available
at \url{https://www.gap-system.org/Packages/packages.html} as a
deposited package. This means that it will be included in future
releases of the GAP distribution.

\appendix
\section{Algorithms for unitary representations} \label{sec:unitary}

During the course of the research for this project, we discovered
several interesting algorithms that apply to the special case when a
representation is unitary. As we have seen in Section \ref{sec:use_basis},
unitary representations have desirable properties that allow us to
perform optimisations that are not possible in general.

In this section, we describe our method to compute a unitary
representation isomorphic to a given representation.

Even further, there is an algorithm due to Dixon \cite{dixon_1970}
that allows the decomposition of a unitary representation into
irreducibles \emph{without} the need to have a complete list of
irreducibles beforehand.

\subsection{Unitarising representations}

Given a representation $\rho : G \to \GL(V)$, define:

\[
S = \sum_{g \in G} \rho(g)\rho(g)^*
\]

To clarify: here, $\rho(g)^*$ means the conjugate transpose of
$\rho(g)$. This is not the same thing as the dual representation,
which we denoted by $\rho^*(g)$, which means $\rho(g^{-1})^T$, the
transpose of $\rho(g^{-1})$. The dual representation does \emph{not}
appear in this section, $^*$ always means ``conjugate transpose''.

If $\rho$ is already unitary, then $S$ is a scalar matrix. Our
strategy is essentially to try to find a change of basis such that $S$
is a scalar matrix, then we will see that this means $\rho$, after
this change of basis, is unitary.

Notice that $S$ is Hermitian, since:

\[
S^* = \sum_{g \in G} (\rho(g)\rho(g)^*)^* = \sum_{g \in G}
\rho(g)\rho(g)^* = S
\]

Notice also that $\rho(g)S\rho(g)^* = S$ by relabelling:

\[
\rho(g)S\rho(g)^* = \sum_{t \in G} \rho(g)\rho(t)\rho(t)^*\rho(g)^* =
\sum_{g \in G} \rho(gt)\rho(gt)^* = \sum_{g \in G} \rho(g)\rho(g)^* = S
\]

We can rephrase the definition of $S$ in terms of a tensor product. We
do this with the goal of using the group sum trick (from
Section \ref{sec:groupsum}) used extensively in our other algorithms.
\begin{align*}
S_{ij} &= \sum_{g \in G} (\rho(g)\rho(g)^*)_{ij} \\
&= \sum_{g \in G} \sum_k \rho(g)_{ik}(\rho(g)^*)_{kj} \\
&= \sum_{g \in G} \sum_k \rho(g)_{ik}(\overline{\rho(g)})_{jk} \\
&= \sum_k \sum_{g \in G} (\rho(g) \otimes \overline{\rho(g)})_{ikjk} \\
&= \sum_k \left(\sum_{g \in G} (\rho(g) \otimes \overline{\rho(g)})\right)_{ikjk} \\
\end{align*}
where $(A \otimes B)_{xyij}$ refers to the $(i,j)$ entry in the
$(x,y)$ block. A tensor product naturally has this block structure,
with each block given by $A_{xy}B$.

The summation $\sum_{g \in G} (\rho(g) \otimes \overline{\rho(g)})$
has summands given by a homomorphism $g \mapsto \rho(g) \otimes
\overline{\rho(g)}$, so we can use the fast group sum trick.

Given a Hermitian matrix $A$, we can decompose $A$ into $A =
LDL^*$ where $L$ is lower triangular with all $1$ on the diagonal, and
$D$ is real, diagonal. We can use the following formulas:
\begin{align*}
D_j &= A_{jj} - \sum_{k=1}^{j-1} L_{jk} L_{jk}^* D_k \\
L_{ij} &= \frac1{D_j} \left(A_{ij} - \sum_{k=1}^{j-1} L_{ik} L_{jk}^*
D_k \right) \text{for } i > j
\end{align*}

We can compute this decomposition for $S$, let $S = LDL^*$. Now:
\[
D = L^{-1}LDL^* (L^*)^{-1} = L^{-1}\rho(g)LDL^*\rho(g)^*(L^*)^{-1} = (L^{-1}\rho(g)L)D(L^{-1}\rho(g)L)^*
\]

Since $D$ is real, we can take its square root. In fact, $D$ has
positive entries, since $S$ is positive definite, which means
$\sqrt{D}$ still has real entries. This lets us do the following:
\begin{align*}
I &= \sqrt{D}^{-1}(L^{-1}\rho(g)L)\sqrt{D}
\sqrt{D}(L^{-1}\rho(g)L)^*\sqrt{D}^{-1} \\
&= ((L\sqrt{D})^{-1}\rho(g)L\sqrt{D})
((L\sqrt{D})^{-1}\rho(g)L\sqrt{D})^* \\
&= \tau(g)\tau(g)^*
\end{align*}

where we define $\tau(g) = (L\sqrt{D})^{-1}\rho(g)L\sqrt{D}$. This
shows that $\tau$ is unitary.

$S$ is invertible since it is a sum of positive definite Hermitian
matrices. This means $L$ is invertible and allows us to define $\tau$
as above. A formula for $L^{-1}$ is given by: $L^{-1} = DL^*S^{-1}$.

The main drawback of this method is that it requires computing square
roots. As we have seen in Section \ref{sec:gram}, square rooting means the
minimal cyclotomic extension required to contain all coefficients can
be very large. This blowup in degree will cause a corresponding blowup
in memory usage and computation time.

\subsection{Decomposing unitary representations}

We describe an algorithm due to Dixon \cite{dixon_1970} for
decomposing unitary representations into irreducibles. A complete
proof of correctness and explanation can be found there, we will focus
on describing the algorithm.

In principle, combined with the method for unitarising a
representation, this gives a way to find the complete irreducible
decomposition of a representation without needing the complete list of
irreducibles.

In practice, this algorithm has drawbacks which will become clear. We
did not implement this algorithm.

The first step is finding a nonscalar, Hermitian element $H$ of the
centraliser ring $C_\rho$. If $E_{rs}$ denotes the matrix with $1$ in
the $(r,s)$ position and $0$ everywhere else, then define:

\[
H_{rs} = \begin{cases}
  E_{rr} & \text{if } r = s \\
  E_{rs} + E_{sr} & \text{if } r > s \\
  i(E_{rs} - E_{sr}) & \text{if } r < s
\end{cases}
\]

For each $r, s$, compute:

\[
H = \frac1{|G|} \sum_{g \in G} \rho(g)^* H_{rs} \rho(g)
\]

$H$ is Hermitian and commutes with the action of $G$ given by
$\rho$, so is an element of the centraliser. Notice that the $H_{rs}$
give a Hermitian basis for $M_n(\mathbb{C})$.

If $H$ is scalar for all $r, s$, by Schur's lemma, this means $\rho$
is in fact irreducible, so we can stop here.

Otherwise, for some $r, s$, $H$ will be nonscalar. $H$ is conjugate
symmetric, hence diagonalisable. We can find an orthonormal
eigenvector basis for $V$ by diagonalising (using e.g. a Jordan form
decomposition), $J = P^{-1}HP$, then orthonormalising the columns of
$P$ using the Gram-Schmidt process. $P$ must be unitary for the
subrepresentations we will find to be unitary.

The action of $G$ preserves the eigenspaces of $H$, so the block
diagonalisation given by $J$ is in fact a decomposition of $V$ into
$G$-invariant subspaces, giving the unitary subrepresentations.

We then recurse until all representations become irreducible.

The downside to this algorithm is the we require $P$ to be unitary. If
it is not, then the subrepresentations are not unitary and we cannot
recurse. Computing $P$ requires an orthonormalisation process which,
when we require exact coefficients, has problematic performance
characteristics when dealing with exact cyclotomic fields (as seen in
Section \ref{sec:gram}).

Another problem is that computing $H$ requires a summation over
$G$. The reason this algorithm is of interest is that it doesn't
require the computation of a complete list of irreducibles for $G$,
which could be intractable if $G$ is large and complex enough. But if
$G$ is large, summing over $G$ is also undesirable.

Lastly, in real world cases, we often want to decompose permutation
representations $\rho : G \to S_n$. These representations are already
unitary, so it may seem as if this algorithm will be a good
choice. However, using Young tableaux, a complete list of irreducibles of
$S_n$ can be computed very easily (as is done using SageMath in
Section \ref{sec:crossing}), while for large $n$, summing over $S_n$
(as required by the unitary algorithm) will be slow.

For these reasons, this algorithm was not the focus of this project.

\bibliographystyle{unsrt}
\bibliography{final}

\end{document}